\documentclass[11pt]{amsart}
\usepackage{amssymb,amsmath,amsthm,latexsym}
\usepackage{mathrsfs}
\newtheorem{lemma}{Lemma}
\newtheorem{theorem}{Theorem}

\newtheorem{proposition}[theorem]{Proposition}
\newtheorem{corollary}[theorem]{Corollary}

\newtheorem{remark}[theorem]{Remark}
\usepackage{graphicx}
\theoremstyle{remark}
\numberwithin{equation}{section}
\newcommand{\Z}{\mathbb{Z}}

\def\Z {\mathbb{Z}}
\def\Q {\mathbb{Q}}

\def\C {\mathbb{C}}

\usepackage{amssymb}

\begin{document}

\title[An asymptotic formula for Goldbach's conjecture]{An asymptotic formula for Goldbach's conjecture with monic polynomials in \(\Z[\theta][x]\)}


\thanks{Correspondence to: Ab\'ilio Lemos}
\thanks{E-mail: abiliolemos@ufv.br}







\maketitle

\begin{center}
{ Ab\'ilio Lemos and A. L. A. de Araujo \\ \small {Universidade Federal de Vi\c{c}osa, CCE, Departamento de Matem\'atica, Vi\c{c}osa, MG, Brasil} }
\end{center}

\begin{abstract}
In this paper, we consider \(D=\mathbb{Z}[\theta]\), where

\[\theta=\left\{\begin{array}{ccc}
\sqrt{-k} & \mbox{if}\;\;\;-k\not\equiv 1 \;(\mbox{mod}\;4)\\ 
	\frac{\sqrt{-k}+1}{2} & \mbox{if}\;\;\;-k\equiv 1 \;(\mbox{mod}\;4)\\ 
	 \end{array}\right.,\] 
\(k\geq 2\) is a squarefree integer, and we proved that the number \(R(y)\) of representations of a monic polynomial \(f(x)\in \Z[\theta][x]\), of degree \(d\geq 1\), as a sum of two monic irreducible polynomials \(g(x)\) and \(h(x)\) in \(\Z[\theta][x]\), with the coefficients of \(g(x)\) and \(h(x)\) bounded in complex modulus by \(y\), is asymptotic to \((4y)^{2d-2}\).

\end{abstract}






\par

\section{Introduction}

Hayes \cite{Hayes1}, in 1965, showed that Goldbach's conjecture is considerably simpler for
polynomials with integer coefficients. In fact,  the following result:

\begin{theorem}
If $f(x)$ is a monic polynomial in $\Z[x]$ with $\partial f= d > 1$, then there
exist monic irreducible polynomials $g(x)$ and $h(x)$ in $\Z[x]$ with the property that
$f(x)=g(x) + h(x)$.
\end{theorem}

In a recent note, Saidak \cite{Saidak}, improving on a result of Hayes, gave Chebyshev-type estimates for the number $R(y)=R_{f}(y)$ of representations of the monic polynomial $f(x) \in \mathbb{Z}[x]$ of degree $d>1$ as a sum of two irreducible monics $g(x)$ and $h(x)$ in $\mathbb{Z}[x]$, with the coefficients of $g(x)$ and $h(x)$ bounded in absolute value by  $y$.

Here, we do not distinguish the sum $g(x)+h(x)$ from $h(x)+g(x)$, and whenever we write that a monic polynomial $p(x) \in \mathbb{Z}[x]$ is "irreducible", we mean irreducible over $\mathbb{Q}$. Saidak's argument with slight modifications gives that, for $y$ sufficiently large,
\[c_1y^{d-1} < R(y)< c_2y^{d-1},\]
where $c_1$ and $c_2$ are constants that depend of the degree and the coefficients of the polynomial $f(x)$. 

More recently, Kozek \cite{Kozek} gave a proof that the number $R(y)$ is asymptotic to $(2y)^{d-1}$, i.e.,
\[\lim_{y\longrightarrow \infty}\frac{R(y)}{(2y)^{d-1}}=1.\]
His approach implies that there is a constant $c_3$ depending only on $d$ such that if $y$ is sufficiently large, then
\[R(y)=(2y)^{d-1} + E, \ \ \textup{where} \ \ |E|\leq c_3y^{d-2}\ln(y).\]  

In $2011$, Dubickas \cite{Dubickas} proved a more general result for the number of representations of $f$
by the sum of $r$ monic irreducible (over $\mathbb{Q}$) integer polynomials $f_1, f_2, ..., f_r$ of
height at most $y$, i.e.,
\[f(x) = f_1(x) + f_2(x) + ... + f_r(x),\]
and, for $r=2$, Dubickas proved that
\begin{equation}\label{eq2}
	R(y)=(2y)^{d-1}+ \mbox{\textit{O}}(y^{d-2})
\end{equation}
for $d\geq 4$,
\begin{equation}\label{eq3}
R(y)=(2y)^{2}+ \mbox{\textit{O}}(y\ln(y))
\end{equation}
for $d=3$, and 
\begin{equation}\label{eq4}
R(y)=2y+ \mbox{\textit{O}}(\sqrt{y})
\end{equation}
for $d=2$. Moreover, for each $d\geq 4$, the error term in (\ref{eq2}) is the best possible for
some $f$. Note that this results improve the error term proved by Kozek \cite{Kozek}. 

In 2013, Dubickas \cite{Dubickas2} proved a necessary and sufficient condition on the list of nonzero integers $u_1, \ldots , u_r$,
$r \geq 2$, under which a monic polynomial $f(x) \in \Z[x]$ is expressible by a linear form $u_1 f_{1}(x) + \cdots + u_r f_{r}(x)$ in
monic polynomials $f_1(x), \ldots , f_r(x)\in \Z[x]$.


We say that $D$ satisfies the property (GC) if 
\[
\begin{array}{l}
\textup{Every element of} \  D[x] \ \textup{of degree} \ d\geq 1 \\
 \textup{can be written as the sum of two
irreducibles in }\,D[x].
\end{array}
\]
When $D=\Z$, we have the Hayes's theorem. Pollack \cite{Pollack} proved the following results:

\begin{proposition}\label{p1}
Suppose that $D$ is an integral domain which is Noetherian and has infinitely
many maximal ideals. Then $D$ has property (GC).
\end{proposition}

\begin{corollary}
If $S$ is any integral domain, then $D = S[x]$ has property (GC).
\end{corollary}

When $D=\mathbb{F}_q$ is a finite fields (not that in this case the assumptions of Proposition \ref{p1} do not holds), Hayes \cite{Hayes2}, in 1966, showed an asymptotic formula to the number $R(f)$ of representation of a odd polynomial $f(x)\in \mathbb{F}_q[x]$ as a sum of three monic irreducible polynomials $g(x)$, $h(x)$ and $t(x)$ in $\mathbb{F}_q[x]$, in the form, $\alpha g(x)+\beta h(x)+\gamma t(x)=f(x)$, where  $\alpha,\beta,\gamma$ in $\mathbb{F}_q^*$ such that $\alpha+\beta+\gamma=f_r$ and $f_r$ is the leading coefficient of $f(x)$.

The asymptotic formula proved by Hayes is
\[R(f)=r^{-3}q^{2r}S(f)+O(q^{1/4}q^{(5/4+\epsilon)r}),\]
where $r=\partial f$, $S(f)$ is the singular series ( see  \cite[eq. 6.10]{Hayes2}), the constant implied in $O$ is independent of both $f$ and $q$. The number $\epsilon$ is defined as $\max\{1/2,\epsilon^*\}$, where $\epsilon^*$ is the least upper bound of the real parts of the zeros of certain $L$-functions. Advances in this direction we can cite the following papers, Pollack \cite{Pollack2, Pollack3}, Webb \cite{Webb} and Car \cite{Mireille1,Mireille2,Mireille3}.  

Following the results of Hayes and Pollack with $D=\Z[\theta]$, where $\theta$ satisfies the properties described in the Lemma \ref{L1} below, and following the ideas of Kozek \cite{Kozek} we will prove that the number $R(y)$ of representation of a monic polynomial $f(x)\in \Z[\theta][x]$ as a sum of two monic irreducible polynomials $g(x)$ and $h(x)$ in $\Z[\theta][x]$, with the coefficients of $g(x)$ and $h(x)$ bounded in complex modulus by $y$, is asymptotic to $(4y)^{2d-2}$.

\section{Notation and Preliminaries Results} \label{ns}
Some well known facts are below.
Let $f(x)=\sum_{i=0}^{d}f_{i}x^{i}$ be a polynomial in $\C[x]$, define 
$$
H(f)=\mbox{max}_{0\leq i\leq d}|f_i|\;\;\; \mbox{and} \;\;\; M(f)=\exp\left({\int_{0}^{1}} \ln|f(e^{2\pi it})|dt \right).
$$  
These expressions $H(f)$ and $M(f)$ will be known as \textit{height} and \textit{Mahler's measure} (see \cite{Mahler1, Mahler2}). Mahler showed that for $0\leq i \leq n$, $|f_i|\leq \displaystyle{n \choose i}M(f)$. Mahler also noted that $M(f)$ is multiplicative. An another results of Mahler is that 

\begin{equation}\label{Mahler1}
\frac{M(f)}{\sqrt{d+1}}\leq H(f)\leq 2^{d-1}M(f).
\end{equation}

An important property of Mahler measure was proved by Landau in \cite{Landau}. Landau showed that 

\begin{equation}\label{Landau}
1\leq M(f)\leq\left(\sum_{i=0}^{d}|f_{i}|^{2}\right)^{\frac{1}{2}}. 
\end{equation}

Now we consider $f(x)=a(x)b(x)$ such that $\partial a=d_1$ and $\partial b=d_2$, i.e., $d=d_1 + d_2$. A direct application of Jensen's formula (see \cite{Mahler1}) results in $H(f)=H(ab)\leq (1+d_1)H(a)H(b)$, if $d_1 \leq d_2$. To do this, let 
$$
f(x)=\sum_{i=0}^{d}f_{i}x^{i}, \;\;\; a(x)=\sum_{k=0}^{d_1}a_{k}x^{k} \;\;\; \mbox{and}\;\;\; b(x)=\sum_{t=0}^{d_2}b_{t}x^{t}.
$$
Then 
$$
|f_i|=|\sum_{j=0}^{d_1}a_jb_{i-j}|\leq (1+d_1)\mbox{max}|a_k|\mbox{max}|b_t|=(1+d_1)H(a)H(b).
$$
Moreover, 

\begin{equation}\label{Mahler2}
H(ab)\geq \frac{H(a)H(b)}{2^{d-2}\sqrt{d+1}}.
\end{equation}

The last inequality follows from the fact that $M(f)$ is multiplicative and from the inequality (\ref{Mahler1}) for $a(x)$ and $b(x)$, i.e., $H(a)\leq 2^{d_{1}-1}M(a)$, $H(b)\leq 2^{d_{2}-1}M(b)$ and $M(ab)\leq H(ab)\sqrt{d+1}$.

Using what has been discussed above we prove the following result.

\begin{lemma}\label{L1}
Let $f(x)=a(x)b(x)$ be a polynomial of degree $d$ in $\Z[\theta][x]$, where $a(x), b(x)\in\Z[\theta][x]$ and
\begin{equation}\label{theta}
\theta=\left\{\begin{array}{ccc}
\sqrt{-k} & \mbox{if}\;\;\;-k\not\equiv 1 \;(\mbox{mod}\;4)\\ 
	\frac{\sqrt{-k}+1}{2} & \mbox{if}\;\;\;-k\equiv 1 \;(\mbox{mod}\;4)\\ 
	 \end{array}\right.,
\end{equation}
where $k$ is a squarefree integer and $k\geq 2$.
Let $a(x)$ and $f(x)$ take the form
$$
a(x)=\sum_{i=0}^{m}a_i x^i \;\;\;\mbox{and}\;\;\; f(x)= \sum_{i=0}^{d}f_i x^i.
$$
Then for $0\leq l\leq m$, $a_l$ satisfies
$$
|a_l|\leq 2^{2d-2}\sqrt{d+1}\left(\sum_{i=0}^{d}|f_{i}|^{2}\right)^{\frac{1}{2}}.
$$
\end{lemma}

\begin{proof}
First, we observe that $|r+s\theta|\geq \frac{1}{2}$, for any $\theta$ above and $r,s\in \Z$. Actually, 

$$
|r+s\theta|=\left\{\begin{array}{ccc}
\sqrt{r^2 +s^2k} & \mbox{if}\;\;\;-k\not\equiv 1 \;(\mbox{mod}\;4)\\ 
	\frac{1}{2}\sqrt{r_{1}^{2} +s^2k}  & \mbox{if}\;\;\;-k\equiv 1 \;(\mbox{mod}\;4)\\ 
	 \end{array}\right.,
$$
where $r_1=2r+s$. But, since $r$ and $s$ are integers and $k > 1$ we have that $\sqrt{r^2 +s^2k}\geq 1$ and $\sqrt{r_{1}^{2} +s^2k}\geq 1$.

From the inequalities (\ref{Mahler1}) and (\ref{Mahler2}) we have 

$$
H(a)H(b)\leq 2^{2d-3}\sqrt{d+1}M(f).
$$
Now, using the inequality (\ref{Landau}) follows that
$$
\frac{1}{2}|a_l|\leq \frac{1}{2}H(a)\leq H(a)H(b)\leq 2^{2d-3}\sqrt{d+1}\left(\sum_{i=0}^{d}|f_{i}|^{2}\right)^{\frac{1}{2}}.
$$
Consequently
$$
|a_l|\leq 2^{2d-2}\sqrt{d+1}\left(\sum_{i=0}^{d}|f_{i}|^{2}\right)^{\frac{1}{2}}.
$$
\end{proof}

Before we describe the next lemma, we will want to define $\mbox{\textit{O}}$-notation. For two
functions $r(y)$ and $\phi(y)$, we will write 
$$r(y) = \mbox{\textit{O}}(\phi(y))$$
as $y \rightarrow\infty$ if and only if there exists a $y_0$ and a $C > 0$ such that
$$|r(y)| \leq C\phi(y)$$
for all $y > y_0$. In the event that a constant $C$ depends only on a value $s$, we will write
$$|r(y)| \leq C_s\phi(y),$$
and also
$$r(y) = \mbox{\textit{O}}_s(\phi(y)).$$
In case that a constant $C$ depends on the coefficients and degree of a polynomial $f(x)$, we use instead $\mbox{\textit{O}}_f$.
\begin{lemma}\label{L2}
Let $d>1$ be an integer and $g_{d-1}\in\Z[\theta]$ is fixed, $\theta$ as in Lemma \ref{L1}. For each $y\geq 2$, let $r_y$ denote the number of $d$-tuples $(g_{d-1}, g_{d-2}, \dots, g_1, g_0)$ of elements in $\Z[\theta]$ satisfying $|g_i| \leq y$ for $i\in \left\{0, 1, \dots, d-1\right\}$ such that the polynomial 
$$
g(x)=\sum_{i=0}^{d-1}g_{i}x^{i} + x^{d}
$$
is reducible. Then $r_y=\mbox{\textit{O}}_g(y^{2d-4}\ln y)$. In particular, $r_y=0$ if $y<|g_{d-1}|$.

\end{lemma}

\begin{proof}
Let $g(x) \in \mathbb{Z}[\theta](x)$ be a reducible, monic polynomial of degree $d>1$ such that all of its coefficients are $\leq y$ in complex modulus and $g_{d-1}$ is fixed as in the lemma. Then exists two monic polynomials $a(x)$ and $b(x) \in \mathbb{Z}[\theta][x]$ of degree $\geq 1$ such that $g(x)=a(x)b(x)$. Let us further take
\[\deg(a) = m \geq n=\deg(b),\]
where $m+n=d$. We write $a(x)$ and $b(x)$ in the following forms:
\[a(x)=x^m+a_{m-1}x^{m-1}+ ... + a_1x + a_0\]
\[b(x)=x^n+b_{n-1}x^{n-1}+ ... + b_1x + b_0.\]
We assert that the number of monic polynomials that we are considering with $g_0=0$ is $\mbox{\textit{O}}_{g}(y^{2d-4})$. Indeed, denoting by $g_j=g_{j,1}+i\sqrt{k}g_{j,2}$ and as $|g_i| \leq y$  we have that $|g_j|=\sqrt{g_{j,1}^2+kg_{j,2}^2}\leq y$. Therefore, the number of possibilities for $g_j$ is bounded by

\begin{equation}\label{eq1}
	(2y+1)(2\frac{y}{\sqrt{k}}+1) - 2\frac{y}{\sqrt{k}} - \left[2\frac{y}{\sqrt{k}} - 2\frac{\sqrt{2y-1}}{\sqrt{k}}\right] -...- \left[2\frac{y}{\sqrt{k}} - 2\frac{\sqrt{y^2-1}}{\sqrt{k}}\right],
\end{equation}
with the sum having $y$ terms, i.e., the term (\ref{eq1}) is $\mbox{\textit{O}}_{k}(y^{2})$ and the assertion follows. 

By the argument above it is sufficient to show that the number of $d$-tuples
\[(a_{m-1},a_{m-2},...,a_1,a_0,b_{n-1}, b_{n-2},..., b_1,b_0)\] 
as above, with $a_0b_0\neq 0$, is equal to $\mbox{\textit{O}}_{g}(y^{2d-4}\ln{y})$.

We consider $a(x)$ which has degree $m\leq d-1$. A similar argument applies to $b(x)$. For $1\leq l \leq m-1$, Lemma \ref{L1} implies 
$$
|a_l|\leq 2^{2d-2}\sqrt{d+1}\left(\sum_{i=0}^{d}|g_{i}|^{2}\right)^{\frac{1}{2}}\leq 2^{2d-2}\sqrt{d+1}\left((d+1)y^{2}\right)^{\frac{1}{2}}=C_dy,
$$
where $C_d$ depends only $d$. Thus, the number of $(d-4)$-tuples
\[(a_{m-2},...,a_1, b_{n-2},..., b_1)\]
is $\mbox{\textit{O}}_{g}(y^{2m-4}y^{2n-4})=\mbox{\textit{O}}_{g}(y^{2d-8})$.

Observe that when we multiply $a(x)$ and $b(x)$, since they are both monic polynomials, the value the coefficient $g_{d-1}$ is the sum $a_{m-1} + b_{n-1}$. Also, recall that $g_{d-1}$ is fixed, so determining $a_{m-1}$ also determines $b_{n-1}$. Hence, the number of $2$-tuples $(a_{m-1}, b_{n-1})$ is $\mbox{\textit{O}}_{k}(y^{2})$.

Since $a_0b_0=g_0$, we have $1\leq |a_0b_0|\leq y$, i.e., 
\[1\leq (a_{0,1}^2+ka_{0,2}^2)(b_{0,1}^2+kb_{0,2}^2)\leq y^2.\]
Thus, the number of $2$-tuples $(a_0,b_0)$ is bounded by
\[
\begin{array}{rcl}
\displaystyle 16\sum_{q\leq y^2}\sum_{\delta|y^2}1 &=&\displaystyle 16\sum_{\delta \leq y^2}\sum_{ \stackrel{q \leq y^2}{\delta|q}}1\\
\displaystyle &\leq&\displaystyle 16\sum_{\delta \leq y^2}\frac{y^2}{\delta} \\
\displaystyle &\leq&\displaystyle 16y^2\sum_{\delta \leq y^2}\frac{1}{\delta} \\
\displaystyle &\leq&\displaystyle 16y^2\left( 1 + \int_1^{y^2}\frac{1}{t}dt\right)\\
\displaystyle &=&\displaystyle \mbox{\textit{O}}(y^2\ln{y^2})= \mbox{\textit{O}}(y^2\ln{y}), 
\end{array}
\]
where the $16$ appears above since each term of $a_0$ and $b_0$ may be either positive  or negative.

Finally, for an integer $d>1$ and a fixed $g_{d-1}\in\Z[\theta]$, the number of $d$-tuples
\[(a_0, a_1, . . . , a_{m-1}, b_0, b_1, . . . , b_{n-1}),\]
corresponding to the coefficients of two monic polynomials $a(x)$ and $b(x)$ in $\Z[\theta][x]$ of degrees
$m$, $n\geq 1$ such that $g(x) = a(x)b(x)$ and the coefficients of $g(x)$ are bounded in absolute value by $y$, as $y \longrightarrow \infty$ is:
\[
\begin{array}{c}
\displaystyle r_y =\mbox{\textit{O}}_{g}(y^{2m-4}y^{2n-4}y^2(\ln{y})y^2) = \mbox{\textit{O}}_{g}(y^{2d-4}\ln{y}),
\end{array}
\]
with the constants depending only on $d$ and $k$.

\end{proof}

\begin{remark}\label{rem0}
If we substitute (\ref{theta}) by
\[
\theta=\left\{\begin{array}{ccc}
\sqrt{k} & \mbox{if}\;\;\;k\not\equiv 1 \;(\mbox{mod}\;4)\\ 
	\frac{\sqrt{k}+1}{2} & \mbox{if}\;\;\;k\equiv 1 \;(\mbox{mod}\;4)\\ 
	 \end{array}\right.,
	\]
	where $k\geq 2$ is a squarefree integer and we consider $\Z[\theta]$ with the usual norm induced by $\Q$, our argument can not be applied because, for $y$ large enough, there are endless possibilities for $g_j=g_{j,1}+\sqrt{k}g_{j,2}$ with $g_{j,1}, g_{j,2} \in \Z$ and satisfying $|g_i| \leq y$. Therefore, we could not get the limitation (\ref{eq1}). 
\end{remark}

\begin{remark}\label{rem1}
If we remove  the condition in Lemma \ref{L2} that $g_{d-1}$ is fixed, then $r_y =\mbox{\textit{O}}_{g}(y^{2d-2}\ln{y})$. This is a direct consequence of the same arguments used to prove (\ref{eq1}). If degree of $g$ is $d-1$ and in this case $g_{d-2}$ is not fixed, then $r_y =\mbox{\textit{O}}_{g}(y^{2d-4}\ln{y})$.
\end{remark}

\begin{lemma}\label{L3}
Let $f(x)$ be a monic polynomial in $\Z[\theta][x]$ of degree $d>1$, such that $f(x)=g(x)+h(x)$, where $g(x)$ or $h(x)$ is a reducible monic polynomial such that $\partial g=d$ and $1\leq\partial h\leq d-1$. Moreover, the coefficients of $g(x)$ and $h(x)$ bounded in complex modulus by $y$. Then the number of pairs $(g(x),h(x))$ where at least one $g(x)$ or $h(x)$ is a reducible monic polynomial is $\mbox{\textit{O}}_{f}(y^{2d-4}\ln y)$.

\end{lemma}

\begin{proof}
First,  we write
\begin{equation}\label{eq4}
f(x)=x^{d} + \sum_{j=0}^{d-1}f_{j}x^{j},\, g(x)=x^{d} + \sum_{j=0}^{d-1}g_{j}x^{j} \,\, \mbox{and}\,\, h(x)=x^{n} + \sum_{j=0}^{n-1}h_{j}x^{j}, 
\end{equation}
where $f_{j}=g_{j}+h_{j}$ and $1\leq n \leq d-1$.

Now, we consider a pairs of monic polynomials $(g(x), h(x))$ where at least one of $g(x)$ or $h(x)$ is reducible. Once a particular $g(x)$ or $h(x)$ is fixed, it determines the other. Thus, we can count separately when $g(x)$ is reducible and when $h(x)$ is reducible. We count the ways that $g(x)$ might be reducible. Since $f(x)=g(x)+h(x)$, by (\ref{eq4}), we have that either $g_{d-1}=f_{d-1}$ or $g_{d-1}=f_{d-1}-1$. Therefore, in any case $g_{d-1}$ is fixed. Now, the coefficients of $g$ are bounded in complex modulus by $y$, and thus by Lemma \ref{L2}, we have that the number of monic reducible polynomials $g(x)$ is $\mbox{\textit{O}}_{f}(y^{2d-4}\ln y)$. Now, we count the ways that $g(x)$ might be reducible. If $\partial h=d-1$, by Remark \ref{rem1}, we have that the number of monic reducible polynomials $h(x)$ is $\mbox{\textit{O}}_{f}(y^{2d-4}\ln y)$ and if $\partial h< d-1$, then the amount of reducible $h(x)$ is smaller, also by Remark \ref{rem1}.

\end{proof}

With the results above we can prove the main result.
\section{Main Result}
\begin{theorem}\label{main}
Let $f(x)$ be a monic polynomial in $\Z[\theta][x]$ of degree $d>1$, $\theta$ as in Lemma \ref{L1}. The number $R(y)$ of representation of $f(x)$ as a sum of two monic irreducible $g(x)$ and $h(x)$ in $\Z[\theta][x]$, with the coefficients of $g(x)$ and $h(x)$ bounded in complex modulus by $y$, is asymptotic to $(4y)^{2d-2}$. 
\end{theorem}

\begin{proof}
Let $f(x) \in \mathbb{Z}[\theta][x]$ be a given monic polynomial of degree $d>1$ that takes the form
$$
x^{d} + \sum_{j=0}^{d-1}f_{j}x^{j}, 
$$
where $f_{j}=f_{j,1}+i\sqrt{k}f_{j,2}$. We are looking for pairs of monic polynomials $g(x), h(x)\in\Z[\theta][x]$ such that $f(x)=g(x)+h(x)$ and the coefficients of $g(x)$ and $h(x)$ are bounded in complex modulus by $y$. Without loss of generality, let $\partial(g)>\partial(h)$, and observe that $\partial(g)=d$ and $1\leq\partial(h)\leq d-1$. In this case,

$$
g(x)=x^{d} + \sum_{j=0}^{d-1}g_{j}x^{j} \,\, \mbox{and}\,\, h(x)=x^{n} + \sum_{j=0}^{n-1}h_{j}x^{j}, 
$$
where $g_{j}=g_{j,1}+i\sqrt{k}g_{j,2}$, $h_{j}=h_{j,1}+i\sqrt{k}h_{j,2}$, $f_{j}=g_{j}+h_{j}$ and $1\leq n \leq d-1$.

If $y\geq 1+\{|f_{0}|, |f_{1}|, \cdots, |f_{d-1}|\}$, then the total number of pairs of monic (not necessarily irreducible) polynomials $g(x), h(x)$ is

$$
\sum_{S=0}^{d-2}\prod_{j=0}^{S}(2\left\lfloor y\right\rfloor+1-|f_{j,1}|)(2\left\lfloor y\right\rfloor+1-|f_{j,2}|)\\
= (4y)^{2d-2} + \mbox{\textit{O}}_{f}(y^{2d-4}),
$$
since $f_{j,1}+i\sqrt{k}f_{j,2}=g_{j,1}+ h_{j,1} +i\sqrt{k}(g_{j,2}+h_{j,2})$.

By Lemma \ref{L3} almost all of these pairs of monic polynomials $g(x)$, $h(x)$ are irreducible. In fact, the number of pairs $(g(x),h(x))$ where at least one $g(x)$ or $h(x)$ is a reducible monic polynomial is $\mbox{\textit{O}}_{f}(y^{2d-4}\ln y)$. Thus,
\[
\begin{array}{rcl}
\displaystyle R(y) & = &\displaystyle \sum_{S=0}^{d-2}\prod_{j=0}^{S}(2\left\lfloor y\right\rfloor+1-|f_{j,1}|)(2\left\lfloor y\right\rfloor+1-|f_{j,2}|)+\mbox{\textit{O}}_{f}(y^{2d-4}\ln y)\\
\displaystyle      & = &\displaystyle ((4y)^{2d-2}+\mbox{\textit{O}}_{f}(y^{2d-4}\ln y))+\mbox{\textit{O}}_{f}(y^{2d-4}\ln y) \\
\displaystyle     & = &\displaystyle (4y)^{2d-2}+\mbox{\textit{O}}_{f}(y^{2d-4}\ln y), 
\end{array}
\] 
where we have used that any constant depending only on the coefficients and degree of $f(x)$ is small compared to $\ln y$ when $y$ is sufficiently large. Therefore,

$$
R(y)\sim (4y)^{2d-2}.
$$

\end{proof}

%
%

\vspace{2cc}


\end{document}